\documentclass{amsart}
\usepackage{url}
\usepackage{amssymb}
\usepackage{mathtools}
\usepackage{comment}
\usepackage[dvipsnames]{color}
\usepackage[all]{xy}
\usepackage{float}
\usepackage{enumitem}

\newtheorem{theorem}[equation]{Theorem}

\newtheorem{prop}[equation]{Proposition}
\newtheorem{lemma}[equation]{Lemma}
\newtheorem{corollary}[equation]{Corollary}
\theoremstyle{definition}

\newtheorem{remark}[equation]{Remark}
\newtheorem{example}[equation]{Example}

\newtheorem{condition}[equation]{Condition}

\theoremstyle{plain}
\numberwithin{equation}{section}
\numberwithin{figure}{section}

\DeclareMathOperator{\torsion}{torsion}

\DeclarePairedDelimiter\abs{\lvert}{\rvert}

\def\a{\alpha}
\def\b{\beta}
\def\c{\gamma}
\def\d{\delta}
\def\e{\epsilon}

\def\C{\mathbb C}

\def\Q{\mathbb Q}

\def\R{\mathbb R}
\def\Z{\mathbb Z}

\def\dd{\textup{d}}

\def\ge{\geqslant}
\def\le{\leqslant}

\def\INTC{K_2^T(C)_{\textup{int}}}

\begin{document}

\title[Mahler measure of polynomials defining genus 2 and 3 curves]{Mahler measure of families of polynomials defining genus 2 and 3 curves}

\author{Hang LIU}
\address{School of Mathematics and Information Science, Shaanxi Normal University, Xi'an 710062, China}
\email{liuhang@snnu.edu.cn}

\author{Hourong Qin}
\address{Department of Mathematics, Nanjing University, Nanjing 210093, China}
\email{hrqin@nju.edu.cn}

\begin{abstract}
In this article, we study the Mahler measures of more than 500 families of reciprocal polynomials defining genus 2 and genus 3 curves. We numerically find relations between the Mahler measures of these polynomials with special values of $L$-functions. We also numerically discover more than 100 identities between Mahler measures involving different families of polynomials defining genus 2 and genus 3 curves.
Furthermore, we study the Mahler measures of several families of nonreciprocal polynomials defining genus 2 curves and numerically find relations between the Mahler measures of these families and special values of $L$-functions of elliptic curves. We also find identities between the Mahler measures of these nonreciprocal families and tempered polynomials defining genus 1 curves. We will explain these relations by considering the pushforward and pullback of certain elements in $K_2$ of curves defined by these polynomials and applying Beilinson's conjecture on $K_2$ of curves. We show that there are two and three explicit linearly independent elements in $K_2$ of certain families of genus 2 and genus 3 curves.
\end{abstract}

\keywords{Mahler measure; special values of $L$-functions; Beilinson's conjecture; genus 2 curves; genus 3 curves; computation}
\subjclass[2010]{11R06,19F27,11Y99}

\maketitle

\section{Introduction}
The (logarithmic) Mahler measure of a rational function $P \in \C[x_1^{\pm 1},\dots,x_n^{\pm 1}]$ is defined by
$$m(P):=\frac{1}{(2\pi i)^n}\int_{\mathbb{T}^n} \log|P(x_1,\dots,x_n)|\frac{dx_1}{x_1}\dots\frac{dx_n}{x_n}$$
where $\mathbb{T}^n = \{|x_1|=1\}\times \dots \times \{|x_n|=1\}$.

Deninger \cite{Den97} found connections between the Mahler measure of some polynomials and certain regulators which are expected to be related to the special value of $L$-functions of elliptic curves. Boyd \cite{Boy98} conducted a systematic numerical study of the Mahler measures of families of two-variable polynomials of genus 1 and some reciprocal polynomials of genus 2 and proposed many conjectures. Note that polynomial of $n$ variables is \emph{reciprocal} if
$$\frac{P(x_1,\dots,x_n)}{P(1/x_1,\dots,1/x_n)}$$
is a monomial $x_1^{b_1}\dots x_n^{b_n}$.
For example, he studied the Mahler measure $m(P_k(x,y))$ of
$$P_k(x,y)=(x+1)y^2+(x^2+(2-k)x+1)y+(x^2+x), \quad k \in \Z$$
and found $m(P_k)/L'(E,0) \in \Q^{\times}$ where $E:P_k(x,y)=0$ is (generically) an elliptic curve and $L'(E,0)$ is the derivative of its $L$-function $L(E,s)$ at $s=0$. These relations are only proved for $k=-4,2$ by Rodriguez-Villegas \cite{RV99}, for $k=-8,1,7$ by Mellit \cite{Me} and for $k=-2,4$ by Rogers and Zudilin \cite{RZ}.

Boyd also considered the Mahler measure of some families related to genus 2 curves. For example, the polynomial
$$Q_k(x,y)=y^2+(x^4+kx^3+2kx^2+kx+1)y+x^4, \quad k \in \Z$$
defines a curve $Z_k$ of genus 2 for $k \notin \{-1,0,4,8\}$. There are two interesting involutions $\sigma_1:(x,y)\mapsto (1/x,1/y)$ and $\sigma_2:(x,y)\mapsto (1/x,y/x^4)$ on $Z_k$. When $Z_k$ is of genus 2, the quotients curves $E_k=Z_k/\langle\sigma_1\rangle$ and $F_k=Z_k/\langle\sigma_2\rangle$ are elliptic curves over $\Q$ whose defining equations could be easily given. There is a natural embedding $Z_k\rightarrow E_k\times F_k$. The Jacobian of $Z_k$ is isogenous to $E_k\times F_k$. Boyd found numerically $m(Q_k)/L'(E_k,0) \in \Q^{\times}$ for $k < 4$ and $k \notin \{-1,0\}$. Bosman \cite{Bos04} studied this family in detail and proved the following relations for the cases $k=-1,2,8$:
\begin{align*}
&m(Q_2)=L'(E_{36},0), \quad \text{where } E_{36}:y^2=x^3+1 (\text{conductor } 36), \\
&m(Q_{-1})=2L'(\chi_{-3},-1) \text{ and } m(Q_8)=4L'(\chi_{-3},-1).
\end{align*}

Accidentally, the elliptic curve $E_k$ is birationally equivalent to $E:P_{k}(x,y)=0$. Since the Mahler measures of $Q_k$ and $P_{k}$ are related to the special values of $L$-functions of $E_k$ and $E$ respectively for suitable $k \in \Z$. One expects a relation between the Mahler measures of $Q_k$ and $P_{k}$. Boyd had the following conjecture
\begin{align}\label{eqn:Boydrelation}
m(Q_{k}) =
\begin{cases}
  2m(P_{k}) \quad 0\le k\le 4,\\
  m(P_{k}) \quad k \le -1,
\end{cases}
\end{align}
which is proved by Bertin, Zudilin \cite{BZ16} and Lalin, Wu \cite{LW} using different methods.

In this article, we extend Boyd's work in two directions. Firstly, we study the Mahler measure of several families of nonreciprocal polynomials defining genus 2 curves and numerically find relations between the Mahler measure of these families and families of polynomials defining genus 1 curves. For example, let
$$R_k(x,y)=x+\frac{1}{x}+y+\frac{1}{y}+(k-4)$$
be Deninger's family, then $R_{k}$ defines an elliptic curve birationally equivalent to $E_2=C/\langle\sigma_2\rangle$. We numerically find the following relation
\begin{align*}
m(Q_{k}(x-1,y)) \stackrel{?}{=}
\begin{cases}
  m(R_{k}) \quad   k\le -1,\\
  \frac{1}{2}(m(R_{k})+m(P_{k})) \quad  k \ge 17
\end{cases}
\end{align*}
where the symbol ``$\stackrel{?}{=}$'' indicates that the identity has been verified to at least 50 decimal places as throughout this article.
We will explain these relations assuming Beilinson's conjecture on $K_2$ of curves and prove that there are two explicit linearly independent elements in $K_2$ of certain genus 2 curves such as the curve defined by $Q_k$.
Secondly, we systematically study reciprocal families of polynomials defining genus 2 and genus 3 curves. We numerically find relations between the Mahler measures of these polynomials with special values of $L$-functions of certain elliptic curves and discover more than 100 identities between Mahler measures involving families of polynomials defining genus 2 and genus 3 curves as Boyd's conjecture \eqref{eqn:Boydrelation}. Let $d_f=L'(\chi_{-f},-1)$ where $\chi_{-f}$ is the real odd primitive character of conductor $f$. We give three new conductors $f=23, 303,755$ of Chinburg's conjecture \cite{Ch}: for any conductor $f$, there is a polynomial $P_f(x,y)$ with integer coefficients for which $m(P_f)/d_f \in \Q^{\times}$ (see Example~\ref{example:chinburg}). We will explain the philosophy behind these relations in terms of Beilinson's conjecture and the regulator theory.

This paper is organized as follows. Section \ref{section:preliminaries} presents the necessary background to understand the relations between the Mahler measures of polynomials and special values of $L$-functions and identities between the Mahler measures involving these polynomials. Section~\ref{section:genus2} and \ref{section:genus3} are devoted to genus 2 and genus 3 case respectively. In Section~\ref{section:genus2}, we consider the Mahler measure of both reciprocal and nonreciprocal polynomials defining genus 2 curves and prove the linear independence of two explicit elements in $K_2$ of certain genus 2 curves. In Section \ref{section:genus3}, we study the Mahler measures of reciprocal polynomials defining genus 3 curves. Section~\ref{section:computingmethod} briefly describe the computing method in this article.

\section{Preliminaries}\label{section:preliminaries}
\subsection{$K_2$ of curves and Beilinson's conjecture}
Bloch \cite{Bl78} considered CM-elliptic curves $E$ over $\Q$, and proved a relation between
a regulator associated to an element in $K_2(E)$ and the value of its $L$-function at 2.
In the meantime, Beilinson \cite{Be85} made very far reaching conjectures on the relation between special values of $L$-functions and regulators of certain $K$-groups of smooth projective varieties defined over number fields. However, computer calculations by Bloch and Grayson \cite{BG} showed that for $K_2$ of curves an additional condition should be added, which led to a modification of Beilinson's conjecture.

Let $ F $ be a field.
Matsumoto's theorem says that the group $K_2(F)$ can be described explicitly as
$$F^* \otimes_\Z F^* \slash \langle a\otimes (1-a), a\in F, a\neq 0,1 \rangle,$$
where $\langle \cdots \rangle$ denotes the subgroup generated by the indicated elements.
The class of $a \otimes b$ is denoted by $\{a,b\}$ and called the Steinberg symbol.

For a smooth projective geometrically irreducible curve $C$ defined over $\Q$ (similar definitions apply for curves over number fields),
the tame $K_2$ group $K_2^T(C)$ of $C$ is the subgroup of $K_2(\Q(C))$ such that the elements have trivial tame symbols. Note that $K_2^T(C)\otimes\Q=K_2(C)\otimes\Q$ by the localization sequence in $K$-theory. The integral tame $K_2$ group $ \INTC $ which is a subgroup of $K_2^T(C)$ plays a key role in the Beilinson's conjecture on $K_2$ of curves (for the definition, see \cite{LJ}). We call its elements \emph{integral}.

Beilinson's conjecture relates $K$-theory of varieties to special values of their $L$-functions via the so-called \emph{regulator}.
There is a well defined pairing between $K_2^T(C)$ and $ H_1(C(\C), \Z)$
giving us the \emph{regulator pairing}
\begin{equation} \label{eqn:pairing}
\begin{split}
\langle \,\cdot\, , \,\cdot\, \rangle : H_1(X;\Z) \times K_2^T(C)/\torsion \to \R
\\
(\c, \a) \mapsto
\frac{1}{2\pi}\int_{\c}\eta(\a)
\,,
\end{split}
\end{equation}
with $ \eta(\a) $ obtained by writing $ \a $ as a sum of symbols
$ \{a,b\} $, and mapping $ \{a,b\} $ to
\begin{equation} \label{eqn:eta}
\eta(\{a,b\}) = \log|a| \dd \arg(b) - \log|b| \dd \arg(a)
\,,
\end{equation}
and $\c$ is
chosen such that $\eta(\a)$ is defined.
Let $ H_1(C(\C), \Z)^+ $ and $ H_1(C(\C), \Z)^- $ be the part
of $ H_1(C(\C), \Z) $ which is invariant and anti-invariant under the action of complex conjugation on $ C(\C) $. It is easy to show that the regulator is trivial on the plus part. Hence we only need to consider the regulator on the minus part.
Suppose that $C$ has genus $g$, if $\c_1,\cdots,\c_g$ form a
basis of $H_1(C(\C);\Z)^-$, and $M_1,\cdots,M_g$ are in
$ K_2^T(C) $, then the regulator $R(M_1,\cdots,M_g)$ is defined by
\begin{equation}\label{eqn:regulator}
R=|\det(\langle \c_i, M_j \rangle)|.
\end{equation}
Beilinson's conjecture expects $ \INTC \otimes_\Z \Q $ to have $ \Q $-dimension
$ g $ and $ R \ne 0 $ if $ M_1,\dots, M_g $ form a basis of
it. Moreover, it expects $ R $ to be rationally proportional to $ L^{(g)}(C,0) $
(see, e.g., \cite[Conjecture~3.11]{Jeu06}).

\subsection{Transfer and restriction homomorphisms between $K_2$ groups}\label{subsection:pushpull} Transfer and restriction homomorphisms between $K$-groups are defined for various kinds of maps between varieties, but we only need the following very special case.
Let $\langle \sigma \rangle$ be an automorphism of order 2 of a curve $C$ and $f:C \rightarrow C/\langle \sigma \rangle$ be the quotient map of curves. There is a transfer homomorphism $f_*$ and a restriction homomorphism $f^*$ between the (integral) tame $K_2$ groups of $C$ and $C/\langle \sigma \rangle$. Suppose $M\in K_2^T(C)$, then
\begin{equation}\label{eqn:pushpullK2}
f^*f_*(M)=M\sigma(M)
\end{equation}
where $\sigma$ acts on $M$ by acting on the functions in the Steinberg symbol.

It is Bosman's insight in \cite{Bos04} that these maps allow us to pushforward the regulator integral from $C$ to $C/\langle \sigma \rangle$.
Let $\gamma \in H_1(C(\C);\Z)^-$. Then we have
\begin{equation}\label{eqn:pushforward}
\int_{\gamma}\eta(f^*f_*(M)) = \int_{f(\gamma)}\eta(f_*(M)).
\end{equation}
This is the key to understand the relations involving the Mahler measures of polynomials defining genus 2 and genus 3 curves in this article.

\subsection{Mahler measure and the regulator}\label{subsection:mahlerandregulator}
Let $C$ be the normalization of the projective closure of the algebraic curve defined by $P(x,y) \in \C[x,y]$, then
$\{x,y\}\in K_2^T(C)\otimes \Q$ is equivalent to $P$ being tempered, i.e., the roots of all the face polynomials of $P$ are roots of unity (see \cite{RV99}).
Denote the degree of $P$ in $y$ by $d$. Write
$$P(x,y)=a_0(x)\prod_{k=1}^d(y-y_k(x))$$
where $y_k(x),k=1,\ldots,d$ are $d$ solutions of $P(x,y)=0$ which maybe chosen to be continuous, piecewise analytic functions of $x$. By Jensen's formula with respect to the variable $y$, we have
\begin{align*}
m(P)-m(a_0(x))&=\frac{1}{(2\pi i)^2}\int_{\mathbb{T}^2} \sum_{k=1}^d\log|y-y_k(x)|\frac{dx}{x}\frac{dy}{y}\\
&=\frac{1}{2\pi i}\int_{\mathbb{T}^1}\sum_{k=1}^d\log^+|y_k(x)|\frac{dx}{x}
\end{align*}
where $\log^+\abs{u}:=\max(\log\abs{u},0)$.

In particular, if $d=2$ and $\abs{y_2(x)}\le 1$ as long as $\abs{x}=1$(this happens if $P$ is reciprocal). Then the above formula can be written as
\begin{align}\label{eqn:eqnmahlerreg}
\begin{split}
m(P)-m(a_0(x))&=\frac{1}{2\pi i}\int_{S}\log\abs{y_1(x)}\frac{dx}{x}.\\
&=\frac{1}{2\pi}\int_{S}\eta(\{x,y\}).
\end{split}
\end{align}
where $S=\{(x,y) : \abs{x}=1,\abs{y_1(x)}\ge1\}$. When $S$ can be seen as a cycle in $H_1(C,\Z)$, then we recover a regulator integral as \eqref{eqn:pairing}. Hence by Beilinson's conjecture, if $C$ is an elliptic curve and $\{x,y\} \in K_2^T(C)_{\textup{int}}\otimes \Q$, the Mahler measure is expected to be rationally proportional to $L'(C,0)$.

\section{Polynomials defining genus 2 curves}\label{section:genus2}
\subsection{Constructing the polynomials}\label{subsection:constructingg2}
Following Boyd's study on the Mahler measure of polynomials defining genus 2 curves in \cite{Boy98}, we consider three types of polynomials $P_k(x,y)$
\begin{align*}
&A(x)y^2+(B(x)+kx^2+kx)y+C(x),\\
&A(x)y^2+(B(x)+k(x^3+x)+lx^2))y+C(x),\\
&A(x)y^2+(B(x)+2a(x^5+x)+k(x^4+x^2)+lx^3)y+C(x).
\end{align*}

We require $P_k(x,y)$ to satisfy the following condition
\begin{condition}\label{condition:Pconditiong2}
$P_k(x,y)$ are tempered, reciprocal and define curves $Z_k$ generically of genus 2. The quotient curve of $Z_k$ by the automorphism $(x,y)\mapsto (1/x,1/y)$ is a curve generically of genus 1.
\end{condition}

$A(x)$ is 1 or a product of cyclotomic polynomials which is reciprocal or antireciprocal. Let $d=3,4$ and $6$ for the first, second and third type respectively. Then $\deg(A(x))\le d$ and $C(x)=x^{d}A(1/x)$.
For the first type $B(x)$ is equal to 0 or $x^3+1$ and for the second type $B(x)$ is equal to 0, $x^4+1$ or $2x^4+2$. For the third type, $\deg(A(x))=6$, $B(x)$ is equal to $2x^6+2$ and $a$ is the coefficient of $x$ in $A(x)$.

Let $B_k(x)$ be the coefficient of $y$ in $P_k(x,y)$ and
$\Delta_k(x)=B_k(x)^2-4A(x)C(x)$
which is reciprocal. For the third type, we have $x^2\mid \Delta_k(x)$ and $\Delta_k(x)/x^2$ is a reciprocal polynomial of degree 8. Let $D_k(x)$ be $\Delta_k(x)$ for the first and second type and be
$\Delta_k(x)/x^2$ for the third type.

Since we require $Z_k$ to be generically of genus 2, there must be a polynomial $f(x)\neq x$ of degree 1 such that $f(x)^2|D_k(x)$ for the second and third type. Since $D_k(x)$ is reciprocal, $\left(xf(1/x)\right)^2$ also divides $D_k(x)$, $f(x)$ must be $x\pm1$. Also $f(x) \nmid A(x)$, since otherwise $f(x)\mid B_k(x)$, $f(x)\mid P(x,y)$.

Completing the square, $P_k(x,y)=0$ can be written as $(2A(x)y+B_k(x))^2=\Delta_k(x).$
Substituting
$$X=\frac{x+1}{x-1}, Y=\frac{2A(x)y+B_k(x)}{x^{\e}f(x)(x-1)^3}$$
where we assume $f(x)=1$ for the first type and $\e=1$ for the third type and $\e=0$ otherwise, the equation reduces to the form $Q(X,Y)=Y^2-h(X^2)=0$ where $h$ is cubic. The automorphism
\begin{equation}\label{eqn:sigma}
\sigma:(x,y)\mapsto (1/x,1/y)
\end{equation}
on the curve defined by $P_k(x,y)$ is transformed to $(X,Y)\mapsto (-X,Y)$ on the curve defined by $Q(X,Y)$ if $f(x)=1$ or $x+1$ and transformed to $(X,Y)\mapsto (-X,-Y)$ if $f(x)=x-1$. So the quotient curve of $Z_k$ by the automorphism is generically an elliptic curve defined by $y^2=h(x)$ if $f(x)=1$ or $x+1$ and $y^2=h^*(x)=x^3h(1/x)$ if $f(x)=x-1$.

For the second and third type, if $\deg(A(x))$ is even, $A(x)$ must be reciprocal since otherwise $(x^2-1) \mid A(x), f(x)\mid A(x)$. If $\deg(A(x))$ is odd, we can assume $A(x)$ to be reciprocal by letting $x\rightarrow -x$, then $f(x)=x-1$ since otherwise $(x+1)\mid A(x)$ and $(x+1)\mid D_k(x)$ imply $(x+1) \mid P_k(x,y)$. So we can always assume $A(x)$ to be reciprocal for the second and third type.

For the second type, we have $f(x)^2\mid D_k(x)$, so $D(-f(0))=B_k(-f(0))^2-4A(-f(0))^2=0$, $B_k(-f(0))=\pm2A(-f(0))$ which gives $l=2f(0)k - B(1) \pm 2A(-f(0))$. On the other hand, if $l=2f(0)k - B(1) \pm 2A(-f(0))$, then $f(x)\mid D_k(x)$. Since $D_k(x)=x^8D(1/x)$, $D(-x)$ is also reciprocal, we can assume $f(x)=x-1$ by letting $x\rightarrow -x$. Then $D_k(x)/f(x)$ is anti-reciprocal, $f(x) \mid (D_k(x)/f(x))$. Hence we have $f(x)^2 \mid D_k(x)$. For the third type, we can get $l=2f(0)k + 4f(0) - 4a \pm 2A(-f(0))$ similarly.

Since we require $P_k(x,y)$ to be tempered, the degree of $A(x)$ must satisfy certain conditions. For example, if $B(x)=0$ for the first type, then $\deg(A(x))=2$ or $3$ since the lattice points corresponding to $kx^2y$ and $kxy$ must be the interior points of the Newton polygon.

We summarize above discussions in the following Proposition.
\begin{prop}\label{prop:constructiong2}
Let notations be as above, $\delta=\pm 1$ and $A(x)$ be reciprocal for the second and third type. Then $P_k(x,y)$ satisfies Condition \ref{condition:Pconditiong2} if $A(x)$ and $B(x)$ satisfy the following:
\begin{enumerate}
  \item for the first type, $x+1\nmid A$ and
  \begin{itemize}
    \item if $B(x)=0$, then $\deg(A)=2$ or $3$;
    \item if $B(x)=x^3+1$, then $\deg(A(x))=0,1$ or $2$;
  \end{itemize}
  \item for the second type, $f(x)\nmid A(x)$, $f(x)=x\pm1$ if $\deg(A(x))$ is even, $f(x)=x-1$ if $\deg(A)$ is odd, $l=2f(0)k - B(1) + \delta 2A(-f(0))$ and
  \begin{itemize}
    \item if $B(x)=0$, then $\deg(A(x))=3$ or $4$;
    \item if $B(x)=x^4+1$, then $\deg(A(x))=1,2,3$ or $4$;
    \item if $B(x)=2x^4+2$, then $\deg(A(x))=4$;
  \end{itemize}
  \item for the third type, $a$ is the coefficient of $x$ in $A(x)$, $\deg(A(x))=6$, $B(x)=2x^6+2$, $f(x)=x\pm1$, $f(x)\nmid A(x)$ and $l=2f(0)k + 4f(0) - 4a + \delta 2A(-f(0))$.
\end{enumerate}
\end{prop}

\begin{remark}\label{remark:g2part}
For the second type, we can assume $f(x)=x-1$ by letting $x\rightarrow -x$ and $k\rightarrow -k$.
Furthermore, if $B(x)=0$, we can assume $l=-2k+ 2A(1)$ by letting $y\rightarrow -y$ and $k\rightarrow -k$.
\end{remark}

Next we discuss the distribution of branch points namely the roots of $\widetilde{D}_k(x)=D_k(x)/f(x)^2$. Note that $\widetilde{d}=\deg(\widetilde{D}_k(x))$ is 5 or 6 and $x\mid D_k(x)$ if $\deg(\widetilde{D}_k(x))$ is 5. Following Boyd's notation in \cite{Boy98}, we say these points have distribution $(a,b,c)$ if there are $a,b$ and $c$ branch points outside, on and inside the unit circle $\abs{x}=1$.

\begin{prop}\label{prop:branchg2}
Let notations be as above and $k \in \R, \abs{k}\gg 0$. Then
\begin{enumerate}
  \item for the first type, the distribution of branch points is $(2,2,\widetilde{d}-4)$;
  \item for the second and third type, the distribution of branch points is $(2,2,\widetilde{d}-4)$ for $f(0)\delta k \ll 0$ and $(3,0,\widetilde{d}-3)$ for $f(0)\delta k\gg 0$;
\end{enumerate}
\end{prop}
\begin{proof}
We prove the Proposition for the second type and $f(x)=x-1$, the other cases are similar.

By Proposition~\ref{prop:constructiong2}, we have $l=-2k -B(1) + 2\delta A(1)$ and
\begin{align*}
D_k(x)&=(B(x)+kx^3+lx^2+kx)^2-4A(x)C(x) \\
&=(kx(x-1)^2+B(x)+(2\delta A(1)-B(1))x^2)^2-4A(x)C(x) \\
&=k^2x^2(x-1)^4+2kx(x-1)^2(B(x)+(2\delta A(1)-B(1))x^2)+R(x).
\end{align*}
where $(x-1)^2 \mid R(x)$. So we have
\begin{equation}\label{eqn:Dtilde}
\widetilde{D}_k(x)=k^2x^2(x-1)^2+2kx(B(x)+(2\delta A(1)-B(1))x^2)+\frac{R(x)}{(x-1)^2}
\end{equation}
where $k^2x^2(x-1)^2$ is the dominant term. By Rouch\'{e}'s theorem, $\widetilde{D}_k(x)$ has two roots around $x=0$ and two roots around $x=1$ . Since $\widetilde{D}_k(x)$ is reciprocal, it also has $\widetilde{d}-4$ roots with big absolute value.

If $\delta k\ll 0$, then $\widetilde{D}_k(1) \sim 4\delta A(1)k<0$ and $\widetilde{D}(x)\sim k^2x^2(x-1)^2>0$ if $x\neq 1$, the two roots of $\widetilde{D}_k(x)$ around 1 are real. On the other hand, if $\delta k\gg 0$, $\widetilde{D}(x)>0$ around 1, the two roots are not real. Since $\widetilde{D}_k(x)$ is reciprocal, if $x$ is a root, then $\bar{x},1/x,1/\bar{x}$ are all roots of $\widetilde{D}_k(x)$. But there are only two roots around 1, the two roots must be on the unit circle.
\end{proof}

\subsection{Integration path and relation between the Mahler measure and special value of $L$-function}\label{subsection:path}
Now we consider the integration path of the Mahler measure of $P_k(x,y)$ as $k\in \R$ and $\abs{k}\gg 0$. The two roots of $P_k(x,y)=0$ are
\begin{equation*}
y_1(x),y_2(x)=\frac{-B_k(x)\pm \sqrt{\Delta_k(x)}}{2A(x)}=\frac{-B_k(x)/x^\frac{d}{2}\pm \sqrt{\Delta_k(x)/x^d}}{2A(x)/x^{\frac{d}{2}}}
\end{equation*}
where $d=3, 4$ and 6 for the first, second and third type respectively. Since $B_k(x)$ and $\Delta_k(x)$ are reciprocal, both $-B_k(x)/x^\frac{d}{2}$ and $\Delta_k(x)/x^d$ are real on the unit circle $\abs{x}=1$.
Note that $\abs{y_1(x)y_2(x)}=\abs{C(x)/A(x)}=1$, we have $\abs{y_1(x)}=\abs{y_2(x)}=1$ if $\Delta_k(x)/x^d \le 0$ which does not contribute to the Mahler measure. If $\Delta_k(x)/x^d>0$, suppose $\abs{y_1(x)}>\abs{y_2(x)}$, then $\abs{y_1(x)}>1>\abs{y_2(x)}$.

Let $f:Z_k \rightarrow Z_k/\langle \sigma \rangle$ be the quotient map between curves
where the quotient curve $E_k=Z_k/\langle \sigma \rangle$ is generically an elliptic curve as in Section \ref{subsection:constructingg2} and $M=\{x,y\}\in K_2^T(Z_k)\otimes\Q$ since $P_k(x,y)$ is tempered. By equation~\eqref{eqn:pushpullK2}, we have
\begin{equation}\label{pushpullM}
f^*f_*(M)=M\sigma(M)=\{x,y\}\left\{\frac{1}{x},\frac{1}{y}\right\}=2M.
\end{equation}
By \eqref{pushpullM} and \eqref{eqn:pushforward}, we can pushforward the regulator integral of $M$ on $Z_k$ to the regulator integral of $f_*(M)$ on $E_k$
$$\frac{1}{\pi}\int_{\c}\eta(M)=\frac{1}{2\pi}\int_{\c}\eta(f^*{f}_*(M))=\frac{1}{2\pi }\int_{f(\c)}\eta({f}_*(M))$$
when $\c$ is a cycle in $H_1(Z_k,\Z)$.
Hence by the discussions in Section \ref{subsection:mahlerandregulator} and Beilinson's conjecture, we can realize the Mahler measure as a regulator integral of $M$ and expect $m(P_k(x,y))/L'(E_k,0) \in \Q^{\times}$ if the following condition is satisfied
\begin{condition}\label{condition:relation}
\begin{enumerate}[label=(\alph*)]
  \item $M \in K_2^T(Z_k)_{\textup{int}}\otimes \Q$;
  \item Let $S=\{(x,y_1(x)) : \abs{x}=1,\abs{y_1(x)}\ge 1\}$. Then $\int_{S}\eta(M)$ is a rational multiple of  $\int_{\c}\eta(\{x,y\})$ where $\c$ is a cycle in $H_1(Z_k,\Z)$.
\end{enumerate}
\end{condition}

One can show as Theorem~8.3 (3) of \cite{Jeu06} that (a) is satisfied for $k \in \Z$. We claim (b) is satisfied if $\abs{k} \gg 0$ for the first type and $f(0)\delta k \ll 0$ for the second and third type.

For example, we consider polynomials of the second type and suppose $f(0)=-1$ and $\d=1$ where $\d$ is as in Proposition~\ref{prop:branchg2}. Then there is a singularity at $(1,-1)$.
If $k\gg 0$, by Proposition~\ref{prop:branchg2}, there are 2 branch points on the unit circle.
By equation \eqref{eqn:Dtilde}, we have $\widetilde{D}_k(-1)>0$ and
$\widetilde{D}_k(1)>0$ for $k\gg 0$. Note that
$$\frac{\Delta_k(x)}{x^4}=\frac{\widetilde{D}_k(x)(x-1)^2}{x^4}=\frac{\widetilde{D}_k(x)}{x^3}\left(x+\frac{1}{x}-2\right).$$ where $(x+1/x-2)<0$ on $\abs{x}=1$ except at $x=1$. Hence on the unit circle $\Delta_k(x)/x^4<0$ around $x=1$ and $\Delta_k(x)/x^4>0$ around $x=-1$. On the arc on $\abs{x}=1$ between the two branch points which does not contain the singularity at $x=1$, we have $\abs{y_1(x)}\ge 1$ and only this part contribute to the Mahler measure.
Let $\c$ be the lift of this arc to $Z_k$. Then $\c \in H_1(Z_k,\Z)$ and $\c$ is the union of $S$ and $\{(x,y_2(x)) : \abs{x}=1,y_2(x)\le 1\}$. Since $\abs{y_1(x)y_2(x)}=1$ on $\abs{x}=1$, we have
$$\int_{\c}\eta(M) = 2\int_{S}\eta(M)$$
with a proper orientation of $\c$.

On the other hand, if $k\ll 0$, $D_k(x)/x^4>0$ on the unit circle except at $x=1$, but $S$ does not correspond to a non-trivial cycle in $H_1(Z_k,\Z)$ since the singularity at $x=1$ blow up.

Hence by the above discussion, when $k\in \Z$, we expect $m(P_k(x,y))/L'(E_k,0) \in \Q^{\times}$ for $\abs{k} \gg 0$ if $P_k(x,y)$ is of the first type and for a semi-infinite interval if $P_k(x,y)$ is of the second and third type. This is compatible with the Boyd's numerical results in \cite{Boy98} and our numerical calculations in \cite{LQGithub}.

\subsection{$K_2$ of families of genus 2 curves}
In Section~\ref{subsection:constructingg2}, we construct families of tempered reciprocal polynomials generically defining genus 2 curves. There is an element $\{x,y\}$ in $K_2^T\otimes\Q$ of these curves. But since these curves are generically of genus 2, by Beilinson's conjecture, we want to explicitly construction another element in $K_2^T\otimes\Q$. Based on the following simple observation, we can achieve this for the families of curves defined by polynomials in Table \ref{tab:polynomialg2}.

\begin{lemma}\label{lemma:tempered}
Let $Q_k(x,y)$ be the tempered families of polynomials in Table \ref{tab:polynomialg2}. $Q_k(x-1,y)$ are also tempered.
\end{lemma}
\begin{proof}
One can check this directly.
\end{proof}

\begin{table}
\caption{Tempered families of polynomials $Q_k(x,y)$ such that $Q_k(x-1,y)$ are also tempered}
\centering
\begin{tabular}{cc}
  \hline
  No. & $Q_k$\\
  \hline
1 & $y^2 + (x^3 + kx^2 + kx + 1)y + x^3$ \\
2 & $(x^2 + x + 1)(y^2+x) + (x^3 + kx^2 + kx + 1)y$ \\
3 & $(x^2 + x + 1)(y^2+x) + (kx^2 + kx)y$ \\
4 & $(x^2+x+1)^2(y^2+1) + (kx^3 + (2k + 2)x^2 + kx)y$ \\
5 & $y^2 + (x^4 + kx^3 + (2k - 4)x^2 + kx + 1)y + x^4$ \\
6 & $y^2 + (x^4 + kx^3 + 2kx^2 + kx + 1)y + x^4$ \\
7 & $(x^2 + x + 1)(y^2+x^2) + (x^4 + kx^3 + (2k - 4)x^2 + kx + 1)y$ \\
8 & $(x^2 + x + 1)(y^2+x^2) + (x^4 + kx^3 + 2kx^2 + kx + 1)y$ \\
9 & $(x^2+x+1)^2(y^2+1) + (x^4 + kx^3 + (2k - 4)x^2 + kx + 1)y$ \\
10& $(x^2+x+1)^2(y^2+1) + (x^4 + kx^3 + 2kx^2 + kx + 1)y$ \\
11& $(x^2+x+1)^2(y^2+1) + (2x^4 + kx^3 + (2k - 6)x^2 + kx + 2)y$ \\
12& $(x^2+x+1)^2(y^2+1) + (2x^4 + kx^3 + (2k - 2)x^2 + kx + 2)y$ \\
13& $(x^2+x+1)^3(y^2+1) + (2x^6 + 6x^5 + kx^4 + (2k-6)x^3 + kx^2 + 6x + 2)y$ \\
14& $(x^2+x+1)^3(y^2+1) + (2x^6 + 6x^5 +kx^4 + (2k-10)x^3 + kx^2 + 6x + 2)y$ \\
  \hline
\end{tabular}
\label{tab:polynomialg2}
\end{table}

This observation is interesting from both the $K$-theoretical perspective which will be demonstrated in Theorem \ref{thm:K2elements} and the Mahler measure perspective which will be explained in the Section \ref{subsection:nonreciprocal}.

Let $Z_k$ be the curve defined by $Q_k(x,y)$. Then we have another element in $K_2^T(Z_k)\otimes\Q$.

\begin{corollary}
$\{x+1,y\}\in K_2^T(Z_k)\otimes \Q$.
\end{corollary}
\begin{proof}
By Lemma \ref{lemma:tempered}, $\{x,y\}$ is in tame $K_2$ of the curve defined by $Q_k(x-1,y)$ modulo torsion. Hence $\{x+1,y\}$ is in $K_2^T(Z_k)\otimes \Q$.
\end{proof}

\begin{theorem}\label{thm:K2elements}
Let all notation be as above.
\begin{enumerate}
\item[(1)] If $k \in \Z$, then $M_1=\{x,y\}$ and $M_2=\{x+1,y\}$ are in $K_2^T(Z_k)_{\textup{int}}\otimes \Q$.
\item[(2)] If $\abs{k} \gg 0$, then $M_1$ and $M_2$ are linearly independent. In particular, if $k \in \Z$ and $\abs{k} \gg 0$, we have two independent elements $M_1$ and $M_2$ in $K_2^T(Z_k)_{\textup{int}}\otimes \Q$.
\end{enumerate}
\end{theorem}
\begin{proof}
The proof of (1) is similar to Theorem~8.3 (3) of \cite{Jeu06}.

We prove (2) for the family defined by $y^2 + (x^3 + kx^2 + kx + 1)y + x^3$. The proof for other families are similar. Let $t=1/k$. The family can be written as $x(x+1)y=-t(y^2+x^3y+y+x^3)$. Let $X$ be the fibred surface defined by this equation. By Lemma 3.1 of \cite{LC}, we can construct two families of closed
loops $\c_{1,t}$ and $\c_{2,t}$ in the fibres $X_t$ around $(x,y,t)=(0,0,0)$ and $(-1,0,0)$.
By Lemma 3.4 of \cite{LC},
$$\lim_{\abs{t}\to 0} \frac{\langle \c_{i,t}, M_j \rangle}{\log \abs{t}} = \pm \delta _{ij}  \quad i,j=1,2. $$
Hence, we have
$$\lim_{\abs{t}\to 0} \frac{\det(\langle \c_{i,t}, M_j \rangle)}{\abs{\log \abs{t}}^{2}} = \pm 1$$
which shows the linear independence of $M_1$ and $M_2$ for $\abs{t} \to 0$, i.e. $\abs{k} \gg 0$.
\end{proof}

\begin{remark}
Theorem \ref{thm:K2elements} also applies to curves defined by nonreciprocal families such as $x(x+1)y^2+(x^3+kx^2+(k-3)x-1)y-x(1+x)$ and $y^2+(-x^3-kx^2+(2-k)x-1)y+x^2$. For $k=1$, the polynomials define the modular curves $X_1(13)$ and $X_1(16)$ respectively and their Mahler measures are studied in \cite{Br}.
\end{remark}

\subsection{The Mahler measure of nonreciprocal families}\label{subsection:nonreciprocal} In this section, we study the Mahler measure of $Q_k(x-1,y)$ where $Q_k(x,y)$ are the families in Table \ref{tab:polynomialg2}. Note that $Q_k(x,y)$ is reciprocal, but $Q_k(x-1,y)$ is nonreciprocal.

As an example, let us first look at the following Boyd's family in detail
\begin{equation}\label{eqn:Boydfamily}
Q_k(x,y)=y^2 + (x^4 + kx^3 + 2kx^2 + kx + 1)y + x^4.
\end{equation}
Completing the square, $Q_k=0$ can be written as
$$y_1^2=x^6 + (2k - 2)x^5 + (k^2 + 3)x^4 + (2k^2 - 4)x^3 + (k^2 + 3)x^2 + (2k - 2)x + 1=\widetilde{D}_k(x)$$
where $y_1=(2y+(x^4 + kx^3 + 2kx^2 + kx + 1))/(x+1)$. Replacing $x$ by $(X+1)/(X-1)$ and $y_1$ by $2Y/(X-1)^3$. The equation becomes
\begin{equation}\label{eqn:xsquareeqn}
Y^2 = (k^2 + k)X^6 + (-2k^2 + 5k + 4)X^4 + (k^2 - 5k + 8)X^2 - k + 4 =h(X^2).
\end{equation}
where $h$ is cubic.
The family of curves $Z_k$ defined by $Q_k(x,y)=0$ is isomorphic to $Z_k'$ defined by \eqref{eqn:xsquareeqn}. The following automorphisms on the curve $Z'_k$
$$\sigma'_1:(X,Y) \mapsto (-X,Y) \quad \text{and} \quad \sigma'_2:(X,Y)\mapsto (-X,-Y)$$
correspond to automorphisms on the curve $Z_k$
$$\sigma_1:(x,y) \mapsto \left(\frac{1}{x},\frac{1}{y}\right) \quad \text{and} \quad \sigma_2:(x,y)\mapsto \left(\frac{1}{x},\frac{y}{x^4}\right).$$
Note that $\sigma_1$ is exactly $\sigma$ in \eqref{eqn:sigma}.

Let $E_k$ and $F_k$ be the elliptic curves $Z_k/\langle \sigma_1 \rangle$ and $Z_k/\langle \sigma_2 \rangle$. $E_k$ and $F_k$ are isomorphic to $Z'_k/\langle \sigma'_1 \rangle$ and $Z'_k/\langle \sigma'_2 \rangle$ which are defined by $y^2=h(x)$ and $y^2=h^*(x)=x^3h(1/x)$ respectively. The Jacobian of $Z_k$ is isogenous to $E_k\times F_k$.

Let $f_1$ and $f_2$ be the quotient maps between $Z_k$ and the quotient curves $E_k,F_k$. By equation~\eqref{eqn:pushpullK2}, we have
\begin{align}
f_1^*{f_1}_*(M_1)&=M_1\sigma(M_1)=\{x,y\}\left\{\frac{1}{x},\frac{1}{y}\right\}=2M_1,\label{eqn:pushpullM1}\\
f_2^*{f_2}_*(M_2)&=M_2\sigma_2(M_2)=\{x+1,y\}\left\{\frac{x+1}{x},\frac{y}{x^4}\right\}=2M_2-M_1.\label{eqn:pushpullM2}
\end{align}

We can pushforward the regulator integral of $M_2=\{x+1,y\}$ on $Z_k$ to $F_k$. By equations \eqref{eqn:pushpullM1},\eqref{eqn:pushpullM2} and \eqref{eqn:pushforward}, we have
\begin{align}
\frac{1}{\pi}\int_{\c}\eta(M_1)&=\frac{1}{2\pi}\int_{\c}\eta(f_1^*{f_1}_*(M_1))=\frac{1}{2\pi }\int_{f_1(\c)}\eta({f_1}_*(M_1)),\label{eqn:pushpullM1reg}\\
\frac{1}{2\pi}\left(2\int_{\c}\eta(M_2)-\int_{\c}\eta(M_1)\right)&=\frac{1}{2\pi}\int_{\c}\eta(f_2^*{f_2}_*(M_2))=\frac{1}{2\pi }\int_{f_2(\c)}\eta({f_2}_*(M_2))\label{eqn:pushpullM2reg}
\end{align}
where $\c$ is a cycle in $H_1(Z_k,\Z)$. $f_1(\c),f_2(\c)$ are cycles in $H_1(E_k,\Z),H_1(F_k,\Z)$ respectively. If $k \in \Z$, by Theorem~\ref{thm:K2elements}, $M_1$ and $M_2$ are in $K_2^T(Z_k)_{\textup{int}}\otimes \Q$. Therefore ${f_1}_*(M_1)$ and ${f_2}_*(M_2)$ are in $K_2^T(E_k)_{\textup{int}}\otimes \Q$ and $K_2^T(F_k)_{\textup{int}}\otimes \Q$ respectively. By Beilinson's conjecture, The right hand side of equations \eqref{eqn:pushpullM1reg} and \eqref{eqn:pushpullM2reg} should be rationally proportional to $L'(E_k,0)$ and $L'(F_k,0)$ respectively. Hence, by equations~\eqref{eqn:pushpullM1reg} and \eqref{eqn:pushpullM2reg}, the regulator integral of $M_2$ on $\c$
\begin{equation}\label{eqn:M2reg}
\frac{1}{2\pi}\int_{\c}\eta(M_2)=\frac{1}{8\pi}\int_{f_1(\c)}\eta({f_1}_*(M_1))+\frac{1}{4\pi}\int_{f_2(\c)}\eta({f_2}_*(M_2))
\end{equation}
should be a rational linear combination of $L'(E_k,0)$ and $L'(F_k,0)$.


Let $y_1(x), y_2(x)$ be the two roots of $Q_k(x,y)=0$. If the following condition is satisfied
\begin{condition}\label{condition:nonreciprocalcondition}
\begin{enumerate}[label=(\alph*)]
  \item $\abs{y_1(x)}\ge 1 \ge\abs{y_2(x)}$ on $\abs{x+1}=1$;
  \item If $x=0$ is not a ramification point, the number of ramification points inside the circle $\abs{x+1}=1$ is even.
\end{enumerate}
\end{condition}
Then the set $\{(x,y_1(x)) : \abs{x+1}=1\}$ is a cycle in $H_1(Z_k,\Z)$. Let $\c$ be this cycle. Applying the translation $x\mapsto x-1$ to $Q_k(x,y), \c$ and $M_2$, then by equation~\eqref{eqn:eqnmahlerreg}, we have
$$m(Q_k(x-1,y))=\frac{1}{2\pi i}\int_{\c}\eta(M_2).$$

Hence $m(Q_k(x-1,y))$ should be a rational linear combination of $L'(E_k,0)$ and $L'(F_k,0)$ if $k\in \Z$ and Condition~\ref{condition:nonreciprocalcondition} is satisfied.

We claim the condition is indeed satisfied for the Boyd's family given by \eqref{eqn:Boydfamily} for $k \in \R, \abs{k}\gg 0$. For condition (a), it is not hard to show
$$\abs{x^4+kx^3+kx^2+kx+1} \ge 1+\abs{x^4}$$
on $\abs{x+1}=1$ and the equality hold only for $x=-1$. Hence for $\abs{x+1}=1$, $y^2+(x^4+kx^3+kx^2+kx+1)y$ is the dominant term on $\abs{y}=1$. By Rouche's theorem, condition (a) holds.
For condition (b), by the proof of Proposition~\ref{prop:branchg2}, $\widetilde{D}_k(x)$ has two roots around $x=-1$, two roots around 0 and two roots with big absolute value. More precisely, we have
$$\widetilde{D}_k(x)=k^2x^2(x+1)^2+2kx(x^4+1)+(x-1)^2(x^2+1)^2$$
and $\widetilde{D}_k(0)=1$. If $k\gg 0$, $\widetilde{D}_k(x)$ change sign when $x<0$, so
the two roots around 0 are negative real roots for $k\gg 0$. Similarly the two roots around 0 are positive real roots for $k\ll 0$. Hence there are four ramification points inside the circle $\abs{x+1}=1$ for $k\gg 0$ and two ramification points inside the circle $\abs{x+1}=1$ which are on the unit circle $\abs{x}=1$ by Proposition~\ref{prop:branchg2} for $k\ll 0$.

Note that by the discussion in Section~\ref{subsection:path}, for $k\ll 0$, we have $\abs{y_1(x)}=\abs{y_2(x)}=1$ on the arc connecting the two ramification points on $\abs{x}=1$ and containing $x=-1$. The cycle $\c$ is homotopic to the path lifted from the arc with proper orientation, so we have
$$\int_{\c}\eta(M_1)=0.$$
Hence we expect $m(Q_k(x-1,y))$ to be rationally proportional to $L'(F_k,0)$ for $k \ll 0$.

We list the numerical relation between $m(Q_k(x-1,y)),L'(E_k,0)$ and $L'(F_k,0)$ for $\abs{k} \le 30$ in Table~\ref{tab:family6}. The numerical results are compatible with above analysis.


\begin{table}
\caption{Relation between $m(Q_k(x-1,y))$ and special values of $L$-functions of $E_k$ and $F_k$. $c_1,c_2,c_3$ are expected to satisfy the relation $c_1m(Q_k(x-1,y))+c_2L'(E_k,0)+c_3L'(F_k,0) = 0$, $N_{E}$ and $N_{F}$ are the conductors of $E_k$ and $F_k$ respectively. For $k=-1$, $E_{-1}$ degenerates, $N_{E}$ is blank.}
\centering
\begin{tabular}{llllll}
\hline
$k$ & $c_1$ & $c_2$ & $c_3$ & $N_{E}$ & $N_{F}$\\
\hline
-30 & -256 & 0 & 1 & 33060 & 38760\\
-29 & 204 & 0 & 1 & 15022 & 35409\\
-28 & -1 & 0 & 3 & 252 & 42\\
-27 & 16 & 0 & 1 & 2730 & 3255\\
-26 & -240 & 0 & 1 & 4420 & 26520\\
-25 & 24 & 0 & 1 & 990 & 4785\\
-24 & 2 & 0 & 1 & 138 & 336\\
-23 & -12 & 0 & 1 & 15686 & 2139\\
-22 & -128 & 0 & 1 & 13860 & 17160\\
-21 & 16 & 0 & 1 & 6090 & 3045\\
-20 & -8 & 0 & 1 & 2660 & 840\\
-19 & -6 & 0 & 1 & 342 & 1311\\
-18 & -24 & 0 & 1 & 2652 & 3432\\
-17 & 12 & 0 & 1 & 170 & 1785\\
-16 & -2 & 0 & 1 & 90 & 240\\
-15 & 40 & 0 & 1 & 4830 & 6555\\
-14 & 16 & 0 & 1 & 4004 & 1848\\
-13 & 24 & 0 & 1 & 1638 & 4641\\
-12 & 1 & 0 & -11 & 660 & 15\\
-11 & 24 & 0 & 1 & 2090 & 3135\\
-10 & -8 & 0 & 1 & 180 & 840\\
-9 & 4 & 0 & 1 & 102 & 663\\
-8 & 1 & 0 & -2 & 14 & 48\\
-7 & 8 & 0 & 1 & 630 & 1155\\
-6 & 8 & 0 & 1 & 420 & 840\\
-5 & -2 & 0 & 1 & 130 & 195\\
-4 & -1 & 0 & 4 & 36 & 24\\
-3 & -2 & 0 & 1 & 66 & 231\\
-2 & -2 & 0 & 1 & 20 & 120\\
-1 & -1 & 0 & 6 &  & 15\\
17 & 8 & 2 & 1 & 306 & 663\\
18 & 48 & -2 & -3 & 1140 & 840\\
19 & 48 & -2 & 1 & 2090 & 3135\\
20 & 24 & -1 & -132 & 1260 & 15\\
21 & 96 & -1 & 2 & 6006 & 4641\\
22 & 96 & -1 & 3 & 7084 & 1848\\
23 & 240 & -10 & 3 & 2070 & 6555\\
24 & -4 & 10 & 1 & 30 & 240\\
25 & 24 & 1 & 1 & 2210 & 1785\\
26 & -48 & 6 & 1 & 468 & 3432\\
27 & -12 & 1 & 1 & 798 & 1311\\
28 & 48 & 1 & -3 & 4060 & 840\\
29 & 864 & 4 & 27 & 18270 & 3045\\
30 & 5376 & 16 & -21 & 20460 & 17160\\
\hline
\end{tabular}
\label{tab:family6}
\end{table}

\begin{remark}
It is interesting to note that $Q_k(x,y)$ and $Q_k(x-1,y)$ define isomorphic curves, but their Mahler measures are expected to be rationally proportional to the $L$-value of the two elliptic factors $E_k$ and $F_k$ respectively for $k \in \Z$ and $k \le -1$.
\end{remark}

The elliptic factor $E_k$ is birational to the genus 1 curve defined by
\begin{equation*}
P_k(x,y)=(x+1)y^2+(x^2+(2-k)x+1)y+(x^2+x).
\end{equation*}
We find that the other elliptic factor $F_k$ is birational to Deninger's family given by
\begin{equation*}
R_k(x,y)=x+\frac{1}{x}+y+\frac{1}{y}+(k-4).
\end{equation*}
By Beilinson's conjecture, $m(P_k)$ and $m(R_k)$ are expected to be rationally proportional to $L'(E_k,0)$ and $L'(F_k,0)$ respectively for $k\in \Z$. Hence $m(Q_k(x-1,y))$ is expected to be a rational linear combination of $m(P_k)$ and $m(R_k)$ for $\abs{k}\gg 0$. Note that the relation between the Mahler measures does not involve arithmetic, it should be valid for $k \in \R$.
We find numerically
\begin{align}\label{eqn:nonreciprocalrelation}
m(Q_{k}(x-1,y)) \stackrel{?}{=}
\begin{cases}
  m(R_{k}) \quad   k\le -1,\\
  \frac{1}{2}(m(P_{k})+m(R_{k})) \quad  k \ge 17.
\end{cases}
\end{align}

\begin{remark}
If one can prove \eqref{eqn:nonreciprocalrelation} which seems to be feasible by extending the ideas in \cite{LW}, then by the evaluation of $m(R_{-1}),m(R_{-4})$ and $m(R_{-12})$ in \cite{La},\cite{LR},\cite{RZ} and \cite{RZ14}, we will have
\begin{align*}
& m(Q_{-1}(x-1,y)) \stackrel{?}{=} 6L'(F_{-1},0),\quad m(Q_{-4}(x-1,y)) \stackrel{?}{=} 4L'(F_{-4},0),\\
& m(Q_{-12}(x-1,y)) \stackrel{?}{=} 11L'(F_{-12},0).
\end{align*}
\end{remark}

The Boyd's family is family 6 in Table \ref{tab:polynomialg2}. We numerically find rational relations between $m(Q_k(x-1,y))$ and the special values of $L$-functions of corresponding elliptic factors as the Boyd's family for families $1,2,5,7,8,11,12$ in Table \ref{tab:polynomialg2}. Please see \cite{LQGithub} for the data of these families. We do not find similar relation for the families $3,4,9,10,13,14$ because they do not satisfy Condition~\ref{condition:nonreciprocalcondition}. For example, we have $y_1(0)=y_2(0)=0$ for family 3, (a) of Condition~\ref{condition:nonreciprocalcondition} does not hold.

We also find tempered families of genus 1 such that these families are birational to the elliptic factors of families $1,2,5,7,8,11,12$ in Table \ref{tab:polynomialg2}. We numerically find similar relations as \eqref{eqn:nonreciprocalrelation}. These relations are summarized in Table \ref{tab:quotientfamilies}. For families 11 and 12, we do not find tempered families of genus 1 which are birational to the quotient curves $Z_k/\langle\sigma_1\rangle$. But we numerically find $2m(Q_k(x-1,y)) \overset{?}{=} m(Q_k(x,y)) + m(R_k)$ for family 11 when $k\ge 14$.
\begin{table}
\caption{No. of $Q_k$ is as in Table \ref{tab:polynomialg2}. $P_k$ and $R_k$ are tempered families birational to the elliptic factors of $Q_k$. Numerically $m(Q_k(x-1,y))\stackrel{?}{=}m(R_k)$ for $k\le k_1$ and $2m(Q_k(x-1,y))\stackrel{?}{=}m(P_k)+m(R_k)$ for $k\ge k_2$.}
\centering
\begin{tabular}{cccc}
  \hline
  No. of $Q_k$  & Defining Polynomials of Quotient Curves & $k_1$ & $k_2$ \\
    \hline
  1 &  $\begin{aligned}
    P_{k} &=y^2 + (2x^2 + (1 - k)x + 1)y + x(x-1)(x+1)^2 \\
    R_{k} &=y^2 + ((k - 3)x + 1)y + x^3
    \end{aligned}$  & 0 & 7 \\
  \hline
  2 &  $\begin{aligned}
    P_{k} &= y^3 + (-x + 3)y^2 - (x^2 + (-k + 1)x - 3)y \\
    &+ (x+1)(x^2+x+1)\\
    R_{k} &= y^2 + (x^2 + (-k + 3)x - 1)y + x^2(x - 1)
    \end{aligned}$ & 2 & 8 \\
  \hline
  5 &  $\begin{aligned}
    P_{k} &= y^2 + (x^2 + (k - 2)x + 1)y + x(x - 1)^2\\
    R_{k} &= y^2 + (4 - k)xy + x(x - 1)^2
    \end{aligned}$ &-8 & 9 \\
  \hline
  6 &  $\begin{aligned}
    P_{k} &= (x + 1)(y^2+x) + (x^2 + (2-k)x + 1)y\\
    R_{k} &= y^2 + (x^2 + (k - 4)x + 1)y + x^2
    \end{aligned}$ &-1 & 17 \\
  \hline
  7 &  $\begin{aligned}
    P_{k} &= y^2 + (k - 2)xy - (x-1)^2(x^2+x+1)\\
    R_{k} &= y^2 + ((k - 4)x + 2)y - (x-1)(x+1)^2
    \end{aligned}$ &-7 & 8 \\
  \hline
  8 & $\begin{aligned}
    P_{k} &=(x^2 + x + 1)(y^2+1) + (x^2 + (-k + 2)x + 1)y \\
    R_{k} &=(x + 1)(y^2+x) + (x^2 + (-k + 4)x + 1)y
    \end{aligned}$ & 1 & 16 \\
  \hline
  11 & $\begin{aligned}
    R_{k} &= y^2 + (k - 8)xy + (x^2+1)^2
    \end{aligned}$ & -4 & \\
  \hline
  12 & $\begin{aligned}
    R_{k} &=(x+1)^2y^2 + (2x^2 + (-k + 8)x + 2)y + (x+1)^2
    \end{aligned}$ & 4 & \\
  \hline
\end{tabular}
\label{tab:quotientfamilies}
\end{table}

\section{Polynomials defining genus 3 curves}\label{section:genus3}
\subsection{Constructing the polynomials}
As in the genus 2 case, we consider three types of polynomials $P_k(x,y)$
\begin{align*}
&A(x)y^2+(B(x)+k(x^4+1)+l(x^3+x^2))y+C(x),\\
&A(x)y^2+(B(x)+k(x^5+x)+l(x^4+x^2)+mx^3)y+C(x),\\
&A(x)y^2+(B(x)+a(x^7+x)+k(x^6+x^2)+l(x^5+x^3)+mx^4)y+C(x).
\end{align*}
We require $P_k(x,y)$ to satisfy the following condition as Condition \ref{condition:Pconditiong2} for the genus 2 families.
\begin{condition}\label{condition:Pconditiong3}
$P_k(x,y)$ are tempered, reciprocal and define curves $Z_k$ generically of genus 3. The quotient curve of $Z_k$ by the automorphism $\sigma: (x,y)\mapsto (1/x,1/y)$ is a curve generically of genus 1.
\end{condition}

$A(x)$ is 1 or a product of cyclotomic polynomials. Let $d=5, 6$ and 8 for the first, second and third type respectively. Then $\deg(A)\le d$ and $C(x)=x^{d}A(1/x)$.
For the first type $B(x)$ is equal to 0 or $x^5+1$ and for the second type $B(x)$ is equal to 0, $x^6+1$ or $2x^6+2$. For the third type, $\deg(A)=8$, $B(x)$ is equal to $2x^8+2$ and $a$ is the coefficient of $x$ in $A(x)$.

As in the genus 2 case, let $B_k(x)$ be the coefficient of $y$ in $P_k(x,y)$ and
$\Delta_k(x)=B_k(x)^2-4A(x)C(x)$
which is reciprocal. For the third type, $x^2\mid \Delta_k(x)$ and $\Delta_k(x)/x^2$ is a reciprocal polynomial of degree 12. Let $D_k(x)$ be $\Delta_k(x)$ for the first and second type and
$\Delta_k(x)/x^2$ for the third type.

If these polynomials define curves generically of genus 3, then there is a polynomial $f(x)$ such that $f(x)^2|D_k(x)$ where $f$ is of degree 1 for the first type and degree 2 for the second and third type respectively.

Completing the square, $P(x,y)=0$ can be written as $(2A(x)y+B_k(x))^2=\Delta_k(x)$.
Substituting
$$X=\frac{x+1}{x-1}, Y=\frac{2A(x)y+B_k(x)}{x^{\e}f(x)(x-1)^4}$$
where $\e=1$ for the third type and $\e=0$ otherwise,
the equation reduces to the form $Q(X,Y)=Y^2-h(X^2)=0$ where $h$ is quartic.

We want the automorphism $\sigma:(x,y)\mapsto (1/x,1/y)$ on $Z_k$ to be transformed to $(X,Y)\mapsto (-X,Y)$ on the curve defined by $Q(X,Y)$ so that $Z_k/\langle \sigma \rangle$ is birational to the curve defined by $y^2=h(x)$ which is generically of genus 1. This requires $f$ to be antireciprocal, hence $f(x)=x-1$ for the first type and $f(x)=x^2-1$ for the second and third type. So $A(x)$ is reciprocal since otherwise $(x-1) \mid A(x)$ and $(x-1) \mid D_k(x)$ which is impossible.

We claim that the degree of $A(x)$ is even. For the first type, if $\deg(A)$ is odd then both $A(x)$ and $B_k(x)$ are divisible by $x+1$ which implies $(x+1)\mid P_k(x,y)$.
For the second and third type, if $A$ is of odd degree then $(x+1)|A$ and $(x+1)^2\mid D_k(x)$, so $(x+1)\mid B_k(x)$ which also implies $(x+1)\mid P_k(x,y)$.

As in the genus 2 families, we can express $l$ and $m$ as linear functions of $k$ by solving $D(1)=0$ for the first type and $D(\pm 1)=0$ for the second and third type. The degree of $A(x)$ also must satisfy certain condition. For example, if $B(x)=0$ for the first type, then $\deg(A)=6$ since $\deg(A(x))$ is even and the lattice points corresponding to $kx^3y,lx^2y,lxy,ky$ must be the interior points of the Newton polygon.

We summarize above discussions in the following Proposition.
\begin{prop}\label{prop:constructiong3}
Let notations be as above. Then $P_k(x,y)$ satisfies Condition \ref{condition:Pconditiong3} if $A$ is reciprocal and of even degree and $A(x), B(x)$ satisfy the following:
\begin{enumerate}
  \item for the first type, $(x+1)\nmid A$, $l=-k-B(1)/2+\delta A(1)$ where $\delta=\pm 1$ and
  \begin{itemize}
    \item if $B=0$, then $\deg(A)=4$;
    \item if $B=x^5+1$, then $\deg(A)=0,2$ or $4$;
  \end{itemize}
  \item for the second type, $A(\pm 1)\neq 0$, $l=\frac{-B(1)+\delta_1A(1)+\delta_2A(-1)}{2}$ and $m=-2k + \delta_1 A(1)- \delta_2 A(-1)$ where $\delta_1=\pm 1, \delta_2=\pm 1$ and
  \begin{itemize}
    \item if $B=0, 2x^6+2$, then $\deg(A)=6$;
    \item if $B=x^6+1$, then $\deg(A)=0,2,4$ or $6$;
  \end{itemize}
  \item for the third type, $a$ is the coefficient of $x$ in $A(x)$, $A(\pm 1)\neq 0$, $l=\frac{-4a+\delta_1A(1)+\delta_2A(-1)}{2}$ and $m=-2k-4 + \delta_1 A(1)-\delta_2A(-1)$ where $\delta_1=\pm 1, \delta_2=\pm 1$.
\end{enumerate}
\end{prop}

\begin{remark}
For the second type, if $B(x)=0$, we can choose $\delta_1=1$ by letting $y\rightarrow -y$ and $k\rightarrow -k$. Furthermore, if $A(x)=A(-x)$ and $l=-B(1)/2$ for the second type then we can choose $m=-2k+2A(1)$ by letting $x\rightarrow -x, k\rightarrow -k$; if $A(x)=A(-x)$ and $m=-2k-4$ for the third type, then $a=0$ and we can choose $l=A(1)$ by letting $x\rightarrow -x$.
\end{remark}

\begin{remark}
Consider the families of genus 3 curves $Z_k$ defined by
\begin{align*}
&y^2+(x^5+1+k(x^4+x)+(-k-1\pm 1)(x^3+x^2))y+x^5,\\
&y^2+(x^6+1+k(x^5+x)-(x^4+x^2)+(-2k+2)x^3)y+x^6,\\
&y^2+(x^6+1+k(x^5+x)-(1 \pm 1)(x^4+x^2)-2kx^3)y+x^6.
\end{align*}
The elements $\{x,y\},\{x-1,y\},\{x+1,y\}$ are in $K_2^T(Z_k)\otimes \mathbb{Q}$. One can show that the elements are linearly independent for $\abs{k}\gg 0$ as in Theorem~\ref{thm:K2elements}.
\end{remark}

As in the genus 2 case, we have the distribution of the branch points namely the roots of $\widetilde{D}(x)=D_k(x)/f(x)^2$ as $k\in \R, \abs{k}\gg 0$ in the following Proposition. Note that $\widetilde{d}=\deg(\widetilde{D}(x))$ is 7 or 8 and $x\mid D(x)$ if $\deg(\widetilde{D}(x))$ is 7.

\begin{prop}\label{prop:branchg3}
Let notations be as above and $k \in \R$. Then
\begin{enumerate}
  \item for the first type, the distribution of branch points is $(2,4,\widetilde{d}-6)$ if $\delta k\gg 0$ and $(3,2,\widetilde{d}-5)$ if $\delta k\ll 0$;
  \item for the second and third type,
  \begin{itemize}
    \item if $\delta_1\delta_2=-1$, the distribution of branch points is $(2,4,\widetilde{d}-6)$ for $\delta_1 k\gg 0$ and $(4,0,\widetilde{d}-4)$ for $\delta_1 k\ll 0$;
    \item if $\delta_1\delta_2=1$, the distribution of branch points is $(3,2,\widetilde{d}-5)$ for $\abs{k} \gg 0$.
  \end{itemize}
\end{enumerate}
\end{prop}
\begin{proof}
The proof is similar to the proof of Proposition~\ref{prop:branchg2}.
\end{proof}

\subsection{Relation between the Mahler measure and special value of $L$-function}
Let $f:Z_k \rightarrow Z_k/\langle \sigma \rangle$ be the map between $Z_k$ and the quotient curve
$Z_k/\langle \sigma \rangle$ which is defined by $y^2=h(x)$ where $h=a_4x^4 + a_3x^3 + a_2x^2 + a_1x + a_0$. It has genus 1 in general and its Jacobian $E_k$ is given by $y^2=x^3+a_2x^2+(a_3a_1-4a_4a_0)x-(4a_4a_2a_0-a_3^2a_0-a_4a_1^2)$ (see \cite[page 57]{Boy98}).

As in the genus 2 case, $M=\{x,y\}\in K_2^T(Z_k)\otimes\Q$ and
$$f^*f_*(M)=M\sigma(M)=\{x,y\}\left\{\frac{1}{x},\frac{1}{y}\right\}=2M.$$
Hence by the discussions in Section~\ref{subsection:path}, we also expect that $m(P_k(x,y))/L'(E_k,0) \in \Q^{\times}$ if Condition~\ref{condition:relation} is satisfied.
For the curves $Z_k$, (a) of Condition~\ref{condition:relation} is satisfied for $k \in \Z$.

By Proposition~\ref{prop:branchg3}, there are either four or two branch points on the unit circle as $\abs{k} \gg 0$. If there are four branch points $P_1,P_2,\overline{P_1},\overline{P_2}$ on $\abs{x}=1$, as in the genus 2 case, one can show
that the two arcs $P_1P_2,\overline{P_1}\overline{P_2}$ contribute to the Mahler measure. Let $\c$ be the lift of the arc $P_1P_2$ to $Z_k$. Then $\c \in H_1(Z_k,\Z)$ and $\c$ is the union of $S=\{(x,y_1(x)) : x\in P_1P_2,y_1(x)\ge 1\}$ and $\{(x,y_1(x)) : x\in P_1P_2,y_2(x)\le 1\}$. Since $\abs{y_1(x)y_2(x)}=1$ on $\abs{x}=1$, we again have
$$\int_{\c}\eta(\{x,y\}) = 2\int_{S}\eta(\{x,y\})$$
with a proper orientation of $\c$. Since the two arcs $P_1P_2$ and $\overline{P_1}\overline{P_2}$ contribute equally to the Mahler measure, (b) of Condition~\ref{condition:relation} is satisfied.

On the other hand, if there are two branch points on $\abs{x}=1$, the arc which contributes to the Mahler measure passes through the singularity. So $y_1(x)$ is not a branch of $y(x)$, (b) does not hold.

Hence by Proposition~\ref{prop:branchg3}, when $k\in \Z$, we expect $m(P_k(x,y))/L'(E_k,0) \in \Q^{\times}$ for a semi-infinite interval if $P_k(x,y)$ is of the first type or if $P_k(x,y)$ is of the second and third type and $\delta_1,\delta_2$ have different sign; we expect $m(P_k(x,y))/L'(E_k,0) \in \Q^{\times}$ for finitely many $k$ if $\delta_1,\delta_2$ have the same sign. Again, this is compatible with our numerical calculation. In the following, we look at several example. For data of these families, please see \cite{LQGithub}.

\begin{example}[$m(P_k(x,y))/L'(E_k,0) \in \Q^{\times}$ for finitely many $k$]\label{example:finiteinterval}
Let
$$P_k(x,y)=y^2+(x^6 + kx^5  - 2kx^3 +  kx + 1)+x^6,$$
we numerically find $m(P_k(x,y))/L'(E_k,0) \in \Q^{\times}$ only for $k=\pm 1,\pm 2$; let
$$P_k(x,y)=(x^2 + x + 1)^{2}(y^2+x^2)+(x^6 + kx^5 + 4x^4 + (-2k +8) x^3 + 4x^2 + kx + 1)y,$$
we find numerically $m(P_k(x,y))/L'(E_k,0) \in \Q^{\times}$ only for $k=0,1$. For these two families, we have $\delta_1=\delta_2=1$ which is compatible with above analysis.
\end{example}

\begin{example}[$m(P_k(x,y))/L'(E_k,0) \in \Q^{\times}$ for a semi-infinite interval $k\in \Z$]\label{example:exampletypee} Let $P_k(x,y)$ be
$$(x^4 + x^3 + x^2 + x + 1)(x^2 - x + 1)(y^2+1)+(x^6 + kx^5 - 2x^4 - (2k + 8) x^3 - 2x^2 + kx + 1)y,$$
then we find numerically $m(P_k(x,y))/L'(E_k,0) \in \Q^{\times}$ for $k\in \Z$ and $k \le -7$. We have $\delta_1=-1,\delta_2=1$ which is again compatible with above analysis. The data for this family is summarized in Table~\ref{tab:exampleinfinite}. It is interesting to note that numerically
$m(P_{0}(x,y))=2/15d_{15}+2/15d_{24}-1/90d_{39}$ which is the linear combination of 3 terms. The discriminant of $P_0(x,y)$ is
$$-3(x-1)^2(x+1)^2(x^2+x+1)^2(x^4 - 2x^3 + 7x^2 - 2x + 1),$$
so there are singularities on $\abs{x}=1$ with
$x=-1,1,(-1\pm\sqrt{-3})/2$
and all these points contribute to the Mahler measure.
\end{example}

\begin{example}[Chinburg's conjecture for $d_{23}, d_{303}$ and $d_{755}$]\label{example:chinburg}
Let
$$P_k(x,y)=(x^4 - x^2 + 1)(y^2+x^2)+(x^6 - kx^5 + 2kx^3 - kx + 1)y,$$
we find numerically
$m(P_6(x,y)) \stackrel{?}{=} d_{23}/6$.
The curve defined by $P_6(x,y)$ has a singularity at $(1,-1)$ and corresponding field of normalization is $\Q(\sqrt{-23})$. One can prove $m(P_6(x,y))/d_{23} \in \Q^{\times}$ by applying Theorem 3 of \cite{BV}.
This gives a new conductor $f=23$ of Chinburg's conjecture. Similarly, let $P_k(x,y)$ be
$$(x^8 + x^7 + x^6 + x^2 + x + 1)(y^2+1)+(2x^8 + 2x^7 -kx^6 + 2x^5 +2kx^4 + 2x^3 -kx^2 + 2x + 2)y,$$
then numerically $m(P_{49}(x,y) \stackrel{?}{=} d_{303}/132$; let $P_k(x,y)$ be
$$(x^8 + x^6 + x^4 + x^2 + 1)(y^2+1)+(2x^8 - kx^6 + 5x^5 + (2k-4)x^4 + 5x^3 - kx^2 + 2)y,$$
then numerically $m(P_{37}(x,y) \stackrel{?}{=} d_{755}/410$. By applying Theorem 3 of loc.\ cit.\ again to these examples, we have two new conductors $303$ and $755$ of Chinburg's conjecture. Note that these examples are in the same shape as Example 8 of \cite{BV}.
\end{example}

\begin{table}
\caption{Data for Example \ref{example:exampletypee}. The $s$ column gives the value of $s\stackrel{?}{=}L'(E_k,0)/m(P_k)$. $N_E$ is the conductor of $E_k$ and the remaining columns give the coefficients of the reduced minimal model, $y^2+a_1xy+a_3y=x^3+a_2x^2+a_4x+a_6$, of $E_k$.}
\centering
\begin{tabular}{|l|l|l|lllll|}
\hline
$k$ & $s$ & $N_{E}$ & $a_1$ & $a_2$ & $a_3$ & $a_4$ & $a_5$\\
\hline
-40&717840&100663680&0&1&0&-59145&3182823\\
-39&1964640&231261030&1&-1&1&-53708&2699727\\
-38&-777600&142895200&0&1&0&-48658&2222688\\
-37&4652160&436187598&1&0&0&-43978&1849364\\
-36&8640&1127808&0&0&0&-2478&23780\\
-35&1041360&137584510&1&0&0&-35650&1236420\\
-34&15882480&1855967520&0&1&0&-31966&989120\\
-33&308160&34595550&1&-1&1&-28580&805047\\
-32&57960&10146752&0&1&0&-25473&620575\\
-31&746520&59849034&1&0&0&-22631&487113\\
-30&-245760&32891040&0&0&0&-20037&373844\\
-29&-1024920&145450370&1&0&0&-17676&278656\\
-28&-12960&1478400&0&1&0&-15533&199563\\
-27&47160&4418622&1&-1&1&-13595&148299\\
-26&2546040&381944992&0&1&0&-11846&94192\\
-25&-725760&60598230&1&0&0&-10275&61425\\
-24&12240&1653120&0&0&0&-8868&35408\\
-23&13680&1682450&1&0&0&-7613&15217\\
-22&386400&47411232&0&1&0&-6498&0\\
-21&156960&11993058&1&-1&1&-5513&-5511\\
-20&-2160&315520&0&1&0&-290&-290\\
-19&-86160&7439070&1&0&0&-3886&-15484\\
-18&-8160&957600&0&0&0&-3225&-16000\\
-17&-282600&30596566&1&0&0&-2653&-15711\\
-16&-3900&331968&0&1&0&-2161&-14833\\
-15&-3240&266310&1&-1&1&-1742&-11811\\
-14&-71280&9449440&0&1&0&-1386&-10640\\
-13&-1200&79950&1&0&0&-1088&-8208\\
-12&-240&25344&0&0&0&-840&-6208\\
-11&7260&527758&1&0&0&-636&-4592\\
-10&-14400&1145760&0&1&0&-470&-3312\\
-9&-900&63630&1&-1&1&-338&-1983\\
-8&-30&3200&0&1&0&-233&-1337\\
-7&1020&48678&1&0&0&-153&-711\\
-4&$m=1/30d_{55}$&&\multicolumn{5}{c|}{}\\
-3&$m=10/3d_{3}$&&\multicolumn{5}{c|}{}\\
0&$m=2/15d_{15}+2/15d_{24}-1/90d_{39}$&&\multicolumn{5}{c|}{}\\
12&$m=1/15d_{15}+1/90d_{183}$&&\multicolumn{5}{c|}{}\\
\hline
\end{tabular}
\label{tab:exampleinfinite}
\end{table}

\subsection{Relations between Mahler measures of families of polynomials}\label{subsection:familyrelation}
Let $Z_k$ be the curve defined by the reciprocal families of genus 2 and 3 studied in this article and $\sigma:(x,y) \mapsto (1/x,1/y)$ be the automorphism of $Z_k$. Then the Jacobian of $Z_k$ has an genus 1 factor $Z_k/\langle \sigma \rangle$ and the Mahler measure of the families are related to the $L$-values of the Jacobian of $Z_k/\langle \sigma \rangle$. If two families have a common factor, we expect the Mahler measures of the families are rationally proportional to each other for $k$ in a suitable range. Boyd \cite{Boy98} numerically found two example relations of this kind which are proved by Bertin, Zudilin \cite{BZ16, BZ17} and Lalin, Wu \cite{LW} using different methods.

We compare the the Mahler measures of the families in this article and tempered genus 1 families and numerically find more than 100 identities of Boyd's type. In the following, we give several examples. For a list of the relations of this kind we find, please see \cite{LQGithub}.

\begin{example}[Relations involving reciprocal families of genus 1] One of Boyd's relations \eqref{eqn:Boydrelation} involves a reciprocal family of genus 1. We find several other relations involving reciprocal family of genus 1. The relations are listed in Table~\ref{tab:relationreciprocalg1}. Note that the first two rows involve genus 2 families, the last two rows involve genus 3 families and seem to be valid only for a finite interval which is compatible with Example~\ref{example:finiteinterval}.
\end{example}

\begin{table}
\caption{Relations between the Mahler measures involving reciprocal families of genus 1.}
\centering
\begin{tabular}{c}
\hline
\(\displaystyle P_k=(x^2
 + x
 + 1)(y^2+x^2)+(x^4
 + k\*x^3
 + 2\*k\*x^2
 + k\*x
 + 1)y\)\\
\(\displaystyle Q_k=(x
 + 1)(x^2
 - x
 + 1)(y^2+x)+(k\*x^3
 + \left(-2\*k
 + 4\right) \*x^2
 + k\*x)y\)\\
\(\displaystyle R_k=(x^2
 + x
 + 1)(y^2+1)+(x^2
 + k\*x
 + 1)y\)\\
\(\displaystyle m(Q_k) = m(P_{-k
 + 2}),k \ge 5 \)\\
\(\displaystyle m(R_k) = m(P_{-k
 + 2}),k \ge 2  \)\\
\(\displaystyle m(R_k) = m(Q_{k}),k \ge 5 \)\\
\hline
\(\displaystyle P_k= (x+1)^2 \*(y^2+x^2)
 + \left(x^4
 + k\*x^3
 + \left(-2\*k
 + 6\right) \*x^2
 + k\*x
 + 1\right) \*y
 \)\\
\(\displaystyle Q_k= (x+1)^3 \*(y^2+x)
 + \left(k\*x^3
 + \left(-2\*k
 + 16\right) \*x^2
 + k\*x\right) \*y
 \)\\
\(\displaystyle R_k= (x+1)^2 \*y^2
 + \left(2\*x^2
 + k\*x
 + 2\right) \*y
\)\\
 \(\displaystyle m(Q_k) = m(P_{k
 - 2}),k \ge 6 \)\\
\(\displaystyle m(R_k) = 1/2m(P_{k
 - 2}),0 \le k \le 4,\quad m(R_k) = m(P_{k
 - 2}),k \ge 5 \)\\
\(\displaystyle m(R_k) = m(Q_{k}),k \ge 6 \)\\
\hline
\(\displaystyle P_k=(x^2
 + x
 + 1)^{2}(y^2+x^2)+(x^6
 + k\*x^5
 + 4\*x^4
 + \left(-2\*k
 + 8\right) \*x^3
 + 4\*x^2
 + k\*x
 + 1)y\)\\
\(\displaystyle Q_k=(x
 + 1)^{2}(y^2+1)+(x^2
 + k\*x
 + 1)y\)\\
\(\displaystyle m(Q_k) = 1/2m(P_{k}),0 \le k \le 2\)\\
\hline
\(\displaystyle P_k=(y^2+x^6)+(x^6
 + k\*x^5
 - 2\*k\*x^3
 + k\*x
 + 1)y\)\\
\(\displaystyle Q_k=(y^2+x^2)+(x^2
 + k\*x
 + 1)y\)\\
\(\displaystyle m(Q_k) = 1/2m(P_{k}),-2 \le k \le 2\)\\
\hline
\end{tabular}
\label{tab:relationreciprocalg1}
\end{table}

\begin{example}[Relation valid for $k\in \R$] Most of the relations are valid for a semi-infinite interval. But we also find relations seem to be valid for $k\in \R$. Let $P_k, Q_k$ be
\begin{align*}
&(x^2
 + x
 + 1)^{2}(x^4
 - x^2
 + 1)(y^2+1)+(2\*x^8
 + 4\*x^7
 + k\*x^6
 + \left(-2\*k
 + 6\right) \*x^4
 + k\*x^2
 + 4\*x
 + 2)y,\\
&(x^2
 + 1)^{2}(x^2
 + x
 + 1)(y^2+1)+(2\*x^6
 + 2\*x^5
 + k\*x^4
 + \left(2\*k
 - 8\right) \*x^3
 + k\*x^2
 + 2\*x
 + 2)y
\end{align*}
respectively, then numerically we find $m(Q_k) \stackrel{?}{=} m(P_{k - 2})$ for $k\in \R$.
\end{example}

\section{Computing and comparing Mahler measures}\label{section:computingmethod}
In this section, we briefly describe how to compute Mahler measures and compare Mahler measures of different families in this article.

We use the standard method described in \cite[Section~3]{BM} to compute the Mahler measure of two-variable polynomials. If the polynomial is tempered reciprocal and of degree 2 in $y$. Let $x=e(t):=\exp(2\pi it)$ on $\abs{x}=1$. The two roots $y_1(t)$ and $y_2(t)$ of $P(e(t),y)$ are given by $(r(t)\pm\sqrt{r(t)^2-s(t)^2})/s(t)$ where $r(t)$ and $s(t)$ are trigonometric polynomials.
Suppose $\abs{y_1(t)}\ge 1$, then
$m(P)$ is the sum of $\int_{\a}^{\b}\log\abs{y_1(t)}dt$ where $[\a,\b]$ are intervals such that the discriminant is positive. Note that $r(t)$ is either positive or negative on each interval since otherwise if $r(t)=0$ the discriminant is negative. Hence we can take $y_1(t)=(r(t)+\sqrt{r(t)^2-s(t)^2})/s(t)$ if $r(t)$ is positive on the interval and $(r(t)-\sqrt{r(t)^2-s(t)^2})/s(t)$ otherwise.

If the polynomial is not reciprocal, there might be more than one $k$ with $\abs{y_k(e(t))} > 1$. In this case, we split the interval by the points $(x,y)$ on $\mathbb{T}^2$ such that $P(x,y)=0$. We can find these points by solving $P(x,y)=0$ and $P(1/x,1/y)=0$. Then in each piece the number of $k$ with $\abs{y_k(e(t))} > 1$ does not change. We integrate on each piece by adding all these $y_k(e(t))$.

Some of the genus 1 families have degree 3 in $y$. For example, we already see a family of this type in the second row of Table \ref{tab:quotientfamilies}. In this case, one combines the numerical solution of $P(e(t),y)=0$ with the numerical integration procedure \texttt{intnum} of PARI \cite{PARI2}.

In this article, we study families of polynomials defining genus 2 and 3 curves. But we also compare the Mahler measures of these families with the Mahler measures of families defining genus 1 curves. There are sixteen unique representatives of convex lattice polygons with a single interior lattice point (see \cite{RV99}) giving tempered genus 1 families. So we determine all the tempered genus 1 families with these Newton polygons by assigning coefficients to the lattice points.

To compare the Mahler measure of different families, we first compute the $j$-invariants of the genus 1 families and genus 1 factors of the genus 2 and 3 families. For the reciprocal families of genus 2 and 3, we only care about the genus 1 factor $Z_k/\langle \sigma \rangle$. The $j$-invariants of these factors are rational functions of $k$. Suppose the $j$-invariants of the genus 1 factors of two families $P_k(x,y)$ and $Q_k(x,y)$ are $j_1(k)$ and $j_2(k)$ respectively. Factorizing the numerator of $j_2(k_1)-j_2(k_2)$, if there are linear factors in the factorization, we find these factors are of the form $k_1\pm k_2 +c$ where $c$ is an integer. So $j_1(k_1)=j_2(k_2)$ when $k_1\pm k_2 +c=0$. This means the two genus 1 factors of curves given by $P_k(x,y)$ and $Q_{\mp(k+c)}(x,y)$ are isomorphic over $\C$ (not necessarily over $\Q$), we numerically compare the Mahler measure between $P_k(x,y)$ and $Q_{\mp(k+c)}(x,y)$ and find the rational relation between them by applying PARI's routine \texttt{bestappr} to the quotient of the Mahler measures of these two families. For the nonreciprocal families in Section \ref{subsection:nonreciprocal}, we find the rational linear relation between the Mahler measures of the families and two quotient families by applying PARI's routine \texttt{lindep}. Similarly, we can also compare the Mahler measure and the corresponding $L$-value by the same routine.

\section{Final remarks}
There are several problems for further study. The most immediate problem is to give a universal algorithm to prove the the conjectural relations between the Mahler measures of different families as in Section \ref{subsection:nonreciprocal} and \ref{subsection:familyrelation}. A possible approach is to extend the ideas of Lalin and Wu in \cite{LW} and the ``parallel lines'' method developed by Mellit in \cite{Me}.

Another direction would be to consider the Mahler measure of polynomials defining curves of genus greater than 3 and find relations between the Mahler measure of different families of polynomials.

It is also desirable to consider the Mahler measure of three variable polynomials and extend Boyd's numerical results in \cite{Boy06}.

\section{Acknowledgements}
The authors would like to thank David Boyd, Francois Brunault, Maltilde Lalin and Wadim Zudilin for very helpful conversations and/or correspondence.
The first author was supported by the Fundamental Research Funds for the Central Universities (No.\,GK201803007) and the National Natural Science Foundation of China (Grant No.\,11726606, and No.\,11801345). The second author was supported by the National Natural Science Foundation of China (Grant No.\,11726605, and No.\,11571163).

\end{document}